\documentclass[reqno, oneside, 12pt]{amsart}

\usepackage[letterpaper]{geometry}
\usepackage{amssymb,amsmath,amsthm,amsxtra,mathtools,mathrsfs,calc,nccmath,bm,color}

\def\Z{\mathbb{Z}}

\def\({\left(}
\def\){\right)}
\def\SL{\operatorname{SL}}

\def\ZL{\Z_{(\ell)}}

\newcommand{\pmfrac}[2]{\left(\frac{#1}{#2}\right)}

\renewcommand{\(}{\left(}
\renewcommand{\)}{\right)}
\usepackage{calc}
\usepackage{graphicx}
\usepackage{caption}
\usepackage{subcaption}

\DeclareMathOperator{\spt}{spt}

\let\temp\phi
\let\phi\varphi
\let\varphi\temp

\newtheorem{theorem}{Theorem}
\newtheorem{lemma}[theorem]{Lemma}

\newtheorem{proposition}[theorem]{Proposition}

\theoremstyle{remark}

\numberwithin{equation}{section}

\begin{document}


\title[Congruences for a mock modular form on $\SL_2(\Z)$ ]{Congruences for a mock modular form on $\SL_2(\Z)$ and the smallest parts function}

\date{\today}
\author{Scott Ahlgren}
\address{Department of Mathematics\\
University of Illinois\\
Urbana, IL 61801} 
\email{sahlgren@illinois.edu} 

\author{Byungchan Kim}
\address{School of Liberal Arts \\ Seoul National University of Science and Technology \\ 232 Gongneung-ro, Nowon-gu, Seoul 01811, Korea}
\email{bkim4@seoultech.ac.kr} 
\thanks{The first author was  supported by a grant from the Simons Foundation (\#426145 to Scott Ahlgren). Byungchan Kim was supported by the Basic Science Research Program through the National Research Foundation of Korea (NRF) funded by the Ministry of Education (NRF-2016R1D1A1A09917344)}

 
\begin{abstract}
Using a family of mock modular forms constructed by Zagier, we study the coefficients of a mock modular form of weight $3/2$ on $\SL_2(\Z)$ modulo primes $\ell\geq 5$.  These coefficients are related to the smallest parts function  of Andrews.
As an application, we reprove a theorem of Garvan regarding the properties of this function modulo $\ell$.
As another application, we  show that congruences modulo $\ell$  for the smallest parts function are rare in a precise sense.

%

 \end{abstract}


\maketitle

\section{Introduction}
Let $\spt(n)$ denote the number of smallest parts in the partitions of $n$.  
This function has been the subject of much recent research.  See, for example, \cite{AA2,  ABL, AK,  And1, AGL, Gar1, Gar2, Ono1} and the references in these papers.
The best known arithmetic properties of $\spt$ are the  congruences of Andrews \cite{And1}:
\begin{equation}\label{eq:Andrews_cong}
\begin{aligned}
\spt \pmfrac{5 n+1}{24} &\equiv 0 \pmod{5}, \\
\spt \pmfrac{7 n+1}{24} &\equiv 0 \pmod{7}, \\
\spt \pmfrac{13 n+1}{24} & \equiv 0 \pmod{13}
\end{aligned}
\end{equation}
(here $\spt(n)$ is defined to be zero if $n$ is not a natural number).
Much of the interest in  $\spt$ arises from the fact that its generating function  is related to a distinguished mock modular form.
In particular,  let $p(n)$ denote the number of partitions of $n$ and define
\begin{equation*}
\begin{aligned}
f(z):=\sum  a(n)q^\frac{n}{24} 
&=  - \sum_{n=-1}^\infty \(12\spt\pmfrac{n+1}{24} + n p\pmfrac{n+1}{24} \) q^\frac{n}{24}\\
&=q^\frac{-1}{24}\(1-35 q-130 q^2-273 q^3+\cdots\)\   \ \ \  q:=e^{2\pi i z}.
\end{aligned}\end{equation*}
Then $f$ is a mock modular form of weight $3/2$ on $\SL_2(\Z)$ (see the next section for details).

Suppose that  $\ell\geq 5$ is prime.
Improving a result  of Bringmann, Garvan and Mahlburg \cite{BGM}, Garvan \cite{Gar1} 
identified each generating function 
\[\sum \spt\pmfrac{\ell n+1}{24}q^\frac{n}{24}\pmod\ell\]
as a modular form modulo $\ell$ of low weight on $\SL_2(\Z)$.  
To state his result, we introduce some notation.
let $\ZL$ denote the set of $\ell$-integral rational numbers,
and for each integer $k$ let $M_k$ denote the space of modular forms on $\SL_2(\Z)$ whose coefficients lie in $\ZL$.
Define the reduction $\overline g$ of 
  $g=\sum b(n)q^\frac{n}{24}\in \ZL[\![q^\frac1{24}]\!]$ coefficientwise,  and  define
  $\overline M_k$ as the set of  reductions of elements of $M_k$.  
   Define the Dedekind eta-function  
\[\eta(z):=q^\frac1{24}\prod_{n=1}^\infty (1-q^n),\]
and for each prime $\ell\geq 5$  define 
\begin{equation}\label{eq:rl_def}
r_\ell\in \{1, \dots, 23\} \ \ \text{ by }\ \   r_{\ell} \equiv -\ell \pmod{24}.
\end{equation}
Garvan  \cite[Corollary 4.2]{Gar2} proved the following 
\begin{theorem}[Garvan]\label{main_thm}
Suppose that   $\ell \geq 5$ is prime. Then  
\[
\sum \overline{\spt\pmfrac{\ell n+1}{24}}q^\frac{n}{24} \in \overline{\eta}^{r_{\ell}}  \overline{M}_{\frac{\ell - r_{\ell}}{2} +1}.
\]
\end{theorem}
Note that, since 
\[M_{\frac{\ell - r_{\ell}}2+1}=\{0\} \iff \ell=5, 7, 13,\]
the congruences \eqref{eq:Andrews_cong}  of Andrews follow  from this result.
Garvan obtains a similar result  for the second rank moment $N_2 (n)$, which is defined by
$\spt (n) = np(n) - \frac12 N_2 (n)$.
His method involves a careful study of modular forms of level $\ell$.

In this paper we take a different approach.
For each prime $\ell$, define the function
\begin{equation}\label{eq:fell2}
f_{\ell} :=\eta \( f -  \pmfrac{-24}{\ell}  \Theta^{\frac{\ell-1}{2}} f \),
\end{equation}
 where  $\Theta$ is the operator defined  by the derivative $\Theta:=\frac{1}{2\pi i}\frac{d}{dz}$. Using a family of mock modular forms constructed by Zagier \cite{Zag1},  we will prove that each $f_\ell$ is congruent to 
a modular form of low weight on $\SL_2(\Z)$.  
\begin{theorem}\label{prop:main_prop}
Suppose that   $\ell \geq 5$ is prime. Then we have
\[
\overline{f_\ell } \in \overline{S}_{\ell+1},
\]
where $S_k$ denotes the space of cusp forms of weight $k$.
\end{theorem}

As an application of Theorem~\ref{prop:main_prop}, we deduce Garvan's Theorem~\ref{main_thm} as a corollary.
As another application, we show that     congruences \eqref{eq:Andrews_cong} of the type found by Andrews are exceedingly rare.
For $\ell\geq 5$ we  say that $\spt$ has a  congruence at $\ell$ if 
\begin{equation}\label{eq:ram_cong}
\spt \pmfrac{\ell n+1}{24}  \equiv 0 \pmod{\ell}\ \ \  \text{for all $n$},
\end{equation}
and we define 
\[w:=\limsup_{X\to\infty}\frac{\#\{\ell\leq X:  \ \spt \ \text{ has a  congruence at $\ell$}\}}{X/\log X}.\]
In \cite{AhlBoy} it was shown that the partition function has a congruence at $\ell$ only if $\ell=5, 7$ or $11$.
For the $\spt$ function we can prove
\begin{theorem}\label{thm:no_cong}   We have  $w=0$.
Moreover, for $\ell<10^{11}$, $\spt$ has a  congruence only at $\ell=5, 7$ and $13$.

\end{theorem}
The bound $10^{11}$  is obtained from a few hours of computation using the first $50$ coefficients of $f$ as described in  Section~\ref{sec:no_cong} and could easily be improved.

Our method uses the properties of a family of mock modular forms on $\SL_2(\Z)$ introduced by Zagier \cite{Zag1} together 
with the theory of modular forms modulo $\ell$.  We begin in the next section by describing these mock modular forms
and developing the necessary background before turning to the proofs in the following sections.


\section{Background}
By work of Bringmann \cite{Bringmann} and   Zagier \cite[\S6]{Zag1} (see also \cite[\S3]{AA2} for example)
it is known that
\begin{equation} \label{eq:fdef}
f(z)=\sum  a(n)q^\frac{n}{24} 
=  - \sum_{n=-1}^\infty \(12\spt\pmfrac{n+1}{24} + n p\pmfrac{n+1}{24} \) q^\frac{n}{24}
\end{equation}
is a mock modular form of weight $3/2$ on $\SL_2(\Z)$ whose multiplier is  conjugate to that of the eta-function. Recall \cite{Zag2} that the $n$-th Rankin-Cohen bracket  is given by 
\[
[g,h]_{n} := \sum_{r=0}^{n} (-1)^r \binom{n+k_1-1}{n-r} \binom{n+k_2 -1}{r} \Theta^{r} g \Theta^{n-r} h,
\]
where $k_1$ and $k_2$ are the respective weights of the modular forms $g$ and $h$.

Let $E_2$ be the usual quasi-modular Eisenstein series of weight $2$ and
for even $k\geq 2$ define
\[
F_{k}  := \sum_{n \neq 0} (-1)^n  \pmfrac{-3}{n-1}  n^{k-1} \frac{q^{n(n+1)/6}}{1-q^n} = - \sum_{r >s> 0}  \pmfrac{12}{r^2-s^2} s^{k-1} q^{rs/6}.
\]
Zagier \cite[\S 6]{Zag1} described a family of mock modular forms on $\SL_2(\Z)$ in every even weight.
\begin{proposition}\label{thm:Zag}
We have
\begin{enumerate}
\item \[12 {F}_{2} +   f \eta   = E_2. \]
\item For all  $n\geq 1$, the function
\[
12 {F}_{2n+2}  + 24^n \binom{2n}{n}^{-1} [f, \eta]_{n}
\]
is a modular form of weight $2n+2$ on $\SL_2(\Z)$.
\end{enumerate}
\end{proposition}
We have corrected a typographical error in the first statement.
The proof, which uses holomorphic projection,  is not given in \cite{Zag1}.  
A sketch of a proof of the first assertion is described in \cite{AA1}.
If
 ${F} = f + f^{-}$ is the completion of the mock modular form $f$, then for $n\geq 1$ we have
\[
\pi_{\operatorname{hol}} ( [{F}, \eta ]_{n} ) = [f,\eta]_n +  \pi_{\operatorname{hol}} ( [f^- ,\eta]_{n} )\in M_{2n+2}.
\]
A description of the holomorphic projection of Rankin-Cohen brackets in weight $(3/2, 1/2)$ is given by 
Mertens \cite[\S5]{Mer} (for these weights there is  quite a bit of simplification).
Zagier's result follows from computing $\pi_{\operatorname{hol}} ( [f^- ,\eta]_{n} )$ explicitly in terms of ${F}_{2n+2}$.

Finally,  we require some basic facts from  the theory of modular forms modulo $\ell$.
Each  $g\in M_k$ has a  filtration
 defined by 
\[
w(g):=\inf \{ k'\ : \ \overline g \in \overline{M}_{k'}\}.
\]
Define the   $U$-operator  by its action on $q$-series:
\[\left(\sum b(n)q^\frac n{24}\right)\big | U_\ell:=\sum b(\ell n)q^\frac n{24}.\]
These  facts about filtrations   can be found in \cite{SD} and  \cite[\S 2.2]{Serre}.
\begin{lemma}\label{filtprop}  If $g\in M_k$ then the following are true.
\begin{enumerate}
\item  $w(g)\equiv k\pmod{\ell-1}$.
\item $\overline{\Theta g}\in \overline M_{k+\ell+1}$.
\item $w(\Theta g)\leq w(g)+\ell+1$, with equality if and only if $w(g)\not\equiv 0\pmod \ell$.
\item $w(g\big |U_\ell) \leq \ell + \frac{ w(g) -1}{\ell}$.
\end{enumerate}
\end{lemma}


\section{Proof of Theorem~\ref{prop:main_prop}}

From the definitions  \eqref{eq:fdef}  and  \eqref{eq:fell2}  we have
\begin{equation}\label{eq:fell}
f_{\ell}   \equiv \eta\(\sum a(n)\left(1-\pmfrac{-n}\ell\right) q^\frac{n}{24}\)\pmod\ell.
\end{equation}
We have $E_2\equiv E_{\ell+1}\pmod \ell$ and $ F_2\equiv  F_{\ell+1}\pmod\ell$.
Set 
\[c(n) := 24^n \binom{2n}{n}^{-1}.\]
From Proposition~\ref{thm:Zag} it follows that 
\begin{equation}\label{eq:rc_cong}
\overline{
f\eta-c\pmfrac{\ell-1}2[ f, \eta]_{\frac{\ell-1}{2}}}\in \overline M_{\ell+1}.
\end{equation}
From the definition we have
\[
[ f, \eta]_\frac{\ell-1}2 := \sum_{r=0}^{\frac{\ell-1}2} (-1)^r \binom{\ell/2}{(\ell-1)/2-r} \binom{\ell/2-1}{r} \Theta^{r} f \, \Theta^{\frac{\ell-1}2-r} \eta.
\]
We have
\[\binom{ \ell /2}{ (\ell-1)/2 -r } \equiv 0 \pmod{ \ell }, \qquad 0\leq r<\frac{\ell-1}2\]
and 
\[\binom{ \ell-1}{(\ell-1)/2}\equiv \binom{(\ell -2)/2}{(\ell-1)/2}\equiv \pmfrac{-1}\ell\pmod\ell.\]
Therefore
\[c\pmfrac{\ell-1}2[ f, \eta]_\frac{\ell-1}2\equiv \pmfrac{-24}{\ell}    \Theta^{\frac{\ell-1}{2}} f  \cdot \eta \pmod{ \ell },
\]
and Theorem~\ref{prop:main_prop} follows from \eqref{eq:rc_cong}, after noting that
\begin{equation}\label{eq:fqexp}
\pmfrac{-24}{\ell} \Theta^{\frac{\ell-1}{2}} f \equiv q^{-1/24} + O \(q^{\frac{23}{24}}\)\pmod\ell.
\end{equation}

\section{Deduction of  Theorem~\ref{main_thm}}

We define $g_\ell $ by
\begin{equation} \label{Eqn:Fell}
g_\ell  := f_\ell \,\Delta^{\frac{\ell^2-1}{24}} , 
\end{equation}
so that $\overline g_\ell\in \overline S_{\ell+1+\frac{\ell-1}2}$.
Using Theorem~\ref{prop:main_prop} and  Lemma~\ref{filtprop}  we find that 
\[
w (g_\ell \big|U_\ell )
\leq \ell +1 +\frac{\ell^2 -1}{2\ell} 
\leq \mfrac{3}{2} \ell +1.
\]
Since $w ( g_\ell ) \equiv 2 \pmod{\ell-1}$, we conclude that
\[
\overline{g_\ell}\big|U_\ell = \overline{ f}\big|{U_\ell}  \cdot \overline{\eta}^{\ell}  \in \overline{S}_{\ell+1}.
\]
Finally, we find that the   $q$-expansion has the form
\[
\overline{g}_{\ell}\big|U_\ell = c\, q^{\frac{r_{\ell} +\ell}{24}} + \cdots
\]
for some $c$. Therefore
\[
\overline{g_\ell}\big|U_\ell \in \overline{\Delta}^{\frac{r_{\ell} +\ell}{24}} \overline{ M}_{\ell+1 - \frac{r_{\ell} + \ell}{2}},
\]
and Theorem~\ref{main_thm} follows.


\section{Proof of Theorem~\ref{thm:no_cong}}\label{sec:no_cong}
Let $\ell\geq 5$ be prime and let $f$ and $f_\ell$ be defined as in \eqref{eq:fdef} and \eqref{eq:fell}.
We begin with a proposition (this can also be deduced from \cite[Thm. 1.1]{Ono1}, \cite[Thm. 1.2]{ABL} or \cite[Cor. 3.2]{AK}).
\begin{proposition} \label{nocong_prop}
If $f\big |U_\ell\equiv 0\pmod \ell$ then $f_\ell\equiv 0\pmod \ell$.
\end{proposition}

\begin{proof}
Suppose that $f\big |U_\ell\equiv 0\pmod \ell$.  Then  $f \equiv \Theta^{\ell-1} f \pmod{\ell}$.
Using this with \eqref{eq:fell} and \eqref{Eqn:Fell} we obtain 
\[\Theta^{\frac{\ell-1}{2}} g_\ell \equiv -\pmfrac{-24}{\ell}  g_\ell \pmod{\ell}.\]
In particular we have 
\[
w ( \Theta^{\frac{\ell-1}{2}} g_\ell )=w (g_\ell).
\]
By Theorem~\ref{prop:main_prop} we have $\overline{g_\ell}\in \overline S_{\frac{\ell^2 +1}{2}+\ell}$.
If it were the case that 
 \[w(g_\ell) =\frac{\ell^2 +1}{2}+\ell\equiv \frac{\ell+1}2\pmod \ell, \]
then Lemma~\ref{filtprop} would give the contradiction
\[w (\Theta^{\frac{\ell-1}{2}} g_\ell ) =  \frac{\ell^2+1}{2}+\ell + \frac{\ell-1}{2} (\ell+1) \neq w (g_\ell).\]
 It follows that
\[
w (g_\ell) \leq   2+ \frac{\ell^2-1}{2}.
\]
By  \eqref{eq:fqexp} we have 
\[
\overline g_\ell = c \, q^{\frac{\ell^2 +23}{24}} + \cdots 
\]
for some  $c$. Since $\dim S_k \leq \frac{\ell^2 -1}{24}$ for $k \leq 2 + \frac{\ell^2-1}{2}$, it follows that 
\[
g_\ell \equiv 0 \pmod{\ell}.
\]
\end{proof}

Finally, we prove Theorem~\ref{thm:no_cong}.
\begin{proof}[Proof of Theorem~\ref{thm:no_cong}]
Given a prime $\ell$, 
it follows from Proposition~\ref{nocong_prop} that if there is an integer $n\equiv 23\pmod{24}$ such that
\begin{equation}\label{eq:cong_test}
\pmfrac{\ell}{n}\neq 1\ \ \text{and}\ \ a(n)\not\equiv 0\pmod \ell,
\end{equation}
then $\spt$ does not have a congruence at $\ell$.
For each  $\ell< 10^{11}$ other than $5$, $7$, and $13$,  we  find an integer $n$ satisfying \eqref{eq:cong_test} among the first $50$ candidates;
this gives the  second assertion
 of Theorem~\ref{thm:no_cong}.  

To prove the first assertion, fix a positive integer $N$, and let $p_1,\dots,  p_N$ be the first $N$ primes 
which are $\equiv 23\pmod{24}$.  Let $E_N$ be the finite set of primes $\ell\geq 5$ which divide $\displaystyle\prod_{1\leq j\leq N} a(p_j)$.
From \eqref{eq:cong_test}, we see that if $\spt$ has a congruence at $\ell$, then either  $\ell\in E_N$ or
$\ell$ is in the the set $Q_N$ defined by the quadratic conditions
\begin{equation}\label{eq:quad_cond}
\pmfrac{\ell}{p_j}=  1,\ \ \  1\leq j\leq N.
\end{equation}
It follows that 
\[\#\{\ell\leq X:  \ \spt \ \text{ has a  congruence at $\ell$}\}\leq \#E_N+\#\{\ell\leq X: \ell\in Q_N\}\sim \frac1{2^N}\frac{X}{\log X}.\]
Therefore $w\leq \frac1{2^N}$.  The theorem follows.
\end{proof}


\bibliographystyle{amsplain}
\bibliography{spt_bib.bib}

\providecommand{\bysame}{\leavevmode\hbox to3em{\hrulefill}\thinspace}
\providecommand{\MR}{\relax\ifhmode\unskip\space\fi MR }
\providecommand{\MRhref}[2]{%
  \href{http://www.ams.org/mathscinet-getitem?mr=#1}{#2}
}
\providecommand{\href}[2]{#2}
\begin{thebibliography}{10}

\bibitem{AA1}
Scott Ahlgren and Nickolas Andersen, \emph{Euler-like recurrences for smallest
  parts functions}, Ramanujan J. \textbf{36} (2015), no.~1-2, 237--248.
  \MR{3296721}

\bibitem{AA2}
\bysame, \emph{Algebraic and transcendental formulas for the smallest parts
  function}, Adv. Math. \textbf{289} (2016), 411--437. \MR{3439692}

\bibitem{AhlBoy}
Scott Ahlgren and Matthew Boylan, \emph{Arithmetic properties of the partition
  function}, Invent. Math. \textbf{153} (2003), no.~3, 487--502. \MR{2000466}

\bibitem{ABL}
Scott Ahlgren, Kathrin Bringmann, and Jeremy Lovejoy, \emph{{$\ell$}-adic
  properties of smallest parts functions}, Adv. Math. \textbf{228} (2011),
  no.~1, 629--645. \MR{2822242}

\bibitem{AK}
Scott Ahlgren and Byungchan Kim, \emph{Mock modular grids and {H}ecke relations
  for mock modular forms}, Forum Math. \textbf{26} (2014), no.~4, 1261--1287.
  \MR{3228930}

\bibitem{And1}
George~E. Andrews, \emph{The number of smallest parts in the partitions of
  {$n$}}, J. Reine Angew. Math. \textbf{624} (2008), 133--142. \MR{2456627}

\bibitem{AGL}
George~E. Andrews, Frank~G. Garvan, and Jie Liang, \emph{Self-conjugate vector
  partitions and the parity of the spt-function}, Acta Arith. \textbf{158}
  (2013), no.~3, 199--218. \MR{3040662}

\bibitem{Bringmann}
Kathrin Bringmann, \emph{On the explicit construction of higher deformations of
  partition statistics}, Duke Math. J. \textbf{144} (2008), no.~2, 195--233.
  \MR{2437679}

\bibitem{BGM}
Kathrin Bringmann, Frank Garvan, and Karl Mahlburg, \emph{Partition statistics
  and quasiharmonic {M}aass forms}, Int. Math. Res. Not. IMRN (2009), no.~1,
  Art. ID rnn124, 63--97. \MR{2471296}

\bibitem{Gar1}
F.~G. Garvan, \emph{Congruences for {A}ndrews' smallest parts partition
  function and new congruences for {D}yson's rank}, Int. J. Number Theory
  \textbf{6} (2010), no.~2, 281--309. \MR{2646759}

\bibitem{Gar2}
\bysame, \emph{Congruences for {A}ndrews' spt-function modulo powers of {$5$},
  {$7$} and {$13$}}, Trans. Amer. Math. Soc. \textbf{364} (2012), no.~9,
  4847--4873. \MR{2922612}

\bibitem{Mer}
Michael~H. Mertens, \emph{Eichler-{S}elberg type identities for mixed mock
  modular forms}, Adv. Math. \textbf{301} (2016), 359--382. \MR{3539378}

\bibitem{Ono1}
Ken Ono, \emph{Congruences for the {A}ndrews spt function}, Proc. Natl. Acad.
  Sci. USA \textbf{108} (2011), no.~2, 473--476. \MR{2770948}

\bibitem{Serre}
Jean-Pierre Serre, \emph{Formes modulaires et fonctions z\^eta {$p$}-adiques},
  (1973), 191--268. Lecture Notes in Math., Vol. 350. \MR{0404145}

\bibitem{SD}
H.~P.~F. Swinnerton-Dyer, \emph{On {$l$}-adic representations and congruences
  for coefficients of modular forms},  (1973), 1--55. Lecture Notes in Math.,
  Vol. 350. \MR{0406931}

\bibitem{Zag2}
Don Zagier, \emph{Modular forms and differential operators}, Proc. Indian Acad.
  Sci. Math. Sci. \textbf{104} (1994), no.~1, 57--75, K. G. Ramanathan memorial
  issue. \MR{1280058}

\bibitem{Zag1}
\bysame, \emph{Ramanujan's mock theta functions and their applications (after
  {Z}wegers and {O}no-{B}ringmann)}, Ast\'erisque (2009), no.~326, Exp. No.
  986, vii--viii, 143--164 (2010), S\'eminaire Bourbaki. Vol. 2007/2008.
  \MR{2605321}

\end{thebibliography}

\end{document}